\documentclass[3p,times]{elsarticle}

\usepackage{amsmath}
\usepackage{amsthm}
\usepackage{booktabs}
\newtheorem{prop}{Proposition}

\newtheorem{defn}{Definition}

\newtheorem{propert}{Properties}
\newtheorem{cor}{Corollary}
\newtheorem{exmp}{Example}
\newtheorem{obs}{Observation}
\newtheorem{rem}{Remark}

\usepackage{amssymb}
\usepackage[figuresright]{rotating}
\usepackage{tcolorbox}

\usepackage[font=footnotesize,labelfont=bf]{caption}
\usepackage[font=footnotesize,labelfont=bf]{subcaption}

\usepackage[english]{babel}
\usepackage{empheq,xcolor}

\begin{document}
	\begin{frontmatter}
		\title{Generating Functions for Lacunary Legendre and Legendre-like Polynomials}

		\author[Enea]{S. Licciardi\corref{cor}}
		\ead{silvia.licciardi@enea.it}
		
		\author[Enea]{G. Dattoli}
		\ead{giuseppe.dattoli@enea.it}
		
		\author[Unict]{R.M. Pidatella}
		\ead{rosa@dmi.unict.it}
		
		\address[Enea]{ENEA - Frascati Research Center, Via Enrico Fermi 45, 00044, Frascati, Rome, Italy}
		\cortext[cor]{Corresponding author: silviakant@gmail.com, silvia.licciardi@enea.it, orcid 0000-0003-4564-8866, tel. nr: +39 06 94005421. }
		\address[Unict]{Dep. of Mathematics and Computer Science, University of Catania, Viale A. Doria 6, 95125, Catania, Italy}
		
		\begin{abstract}
		The use of operational methods of different nature is shown to be a fairly powerful tool to study different problems regarding the theory of Legendre and Legendre-like polynomials. We show how the use of the well known integral representations linking Hermite and Legendre like polynomials and of operational technique allow the derivation of new properties regarding the generating functions, repeated derivatives of irrational functions and so on. We finally extend the method of the integral transform to functions with more than one variable and derive new expansion criteria of a given function in terms of Legendre polynomials.
		\end{abstract}
		
		
	\end{frontmatter}
	
	\section{Introduction}

The theory of special polynomials has benefitted in the past from different techniques, including algebraic \cite{Dattoli} and umbral method \cite{S.Roman}. In this paper we go back to the integral representation method \cite{L.C.Andrews} which allows the framing of different special polynomials and functions, as well, in terms of multivariable Hermite polynomials \cite{Babusci}. \\

\noindent The idea put forward in \cite{Babusci} is that of exploiting the wealth of properties of Hermite polynomial families to derive old and new identities regarding other polynomials belonging to standard and generalized forms of Legendre and Legendre like type. We will in particular deal with Chebyshev, Legendre, Jacobi{\dots} \cite{DattLoren}.\\

 The paper is intended to provide a comprehensive description of the use of the integral transform method and of the relevant link with the previously mentioned algebraic and umbral techniques.\\

  We start by providing an example of interplay between two variable Chebyshev polynomials of the second kind $U_{n} (x,y)$ and Hermite $H_{n} (x,y)$ polynomials ($HP$).
  
  \begin{exmp}
   We consider the integral representation for Chebyshev polynomials of the second kind $U_{n} (x,y)$ \cite{L.C.Andrews}
  
\begin{equation} \label{GrindEQ__1_} 
U_{n} (x,y)=\frac{1}{n!} \int _{0}^{\infty }e^{-s}  H_{n} (-s\, x,-sy)ds , \quad \forall x,y\in\mathbb{R}, \forall n\in\mathbb{N}.
\end{equation} 
Recalling $HP$ serie representation and its generating function \cite{Appel}

\begin{equation}\begin{split} \label{GrindEQ__2_} 
& H_{n} (x,y)=n!\, \sum _{r=0}^{\lfloor\frac{n}{2}\rfloor }\frac{x^{n-2\, r} y^{r} }{(n-2\, r)!\, r!}  ,\\ 
& \sum _{n=0}^{\infty }\frac{t^{n} }{n!}  H_{n} (x,y)=e^{x\, t+y\, t^{2} } 
\end{split} \end{equation} 
and using the definition of the Euler gamma function

\begin{equation} \label{GrindEQ__3_} 
\Gamma (\nu )=\int _{0}^{\infty }e^{-s}  s^{\nu -1} ds, \qquad Re(\nu)>0,
\end{equation} 
we end up with the explicit definition of $U_{n} (x,y)$ as

\begin{equation}\label{key}
U_{n} (x,y)=(-1)^{n} \sum _{r=0}^{\lfloor\frac{n}{2} \rfloor}\frac{(n-r)!\, x^{n-2r} (-y)^{r} }{(n-2\, r)!\, r!}.  
\end{equation}
 Multiplying both sides of eq. \eqref{GrindEQ__1_} by $t^{n} $ and then by summing up over the index $n$, we obtain the well known generating function of second kind Chebyshev polynomials, namely
 
\begin{equation}\begin{split}\label{GrindEQ__5_} 
 \sum _{n=0}^{\infty }t^{n}  U_{n} (x,y)&=\sum _{n=0}^{\infty }\frac{t^{n} }{n!} \int _{0}^{\infty }e^{-s}  H_{n}  (-s\, x,-s\, y)ds=\\
 & =\int _{0}^{\infty }e^{-s\, \left(1+xt+yt^{2} \right)}  ds=\frac{1}{1+x\, t+y\, t^{2} } ,\qquad \quad Re(1+x\, t+y\, t^{2})>0 .
\end{split}\end{equation} 
By using furthermore the generating function \cite{DattLoren}

\begin{equation} \label{GrindEQ__6_} 
\sum _{n=0}^{\infty }\frac{t^{n} }{n!}  H_{n+l} (x,y)=H_{l} (x+2\, y\, t,\, y)\, e^{x\, t+y\, t^{2} }  
\end{equation} 
and employing the same procedure as before, we also easily find that, $\forall \mid t \mid <1,\;\; Re(1+x\, t+y\, t^{2})>0$, 

\begin{equation}\label{GrindEQ__7_} 
 \sum _{n=0}^{\infty }t^{n}  \frac{(n+l)!}{n!} U_{n+l} (x,y)=l!\;\dfrac{U_{l} \left(\dfrac{x+2\, y\, t}{1+x\, t+y\, t^{2} } ,\, \dfrac{y}{1+x\, t+y\, t^{2} } \right)}{\left(1+x\, t+y\, t^{2} \right)}
 \end{equation} 
and indeed we get from eq. \eqref{GrindEQ__1_}

\begin{equation} \label{GrindEQ__8_} 
\sum _{n=0}^{\infty }t^{n}  \frac{(n+l)!}{n!} U_{n+l} (x,y)=\int _{0}^{\infty }e^{-s\, (1+xt+yt^{2} )}  H_{l} (-(x+2\;y\, t)s,-sy)\, ds ,
\end{equation} 
which after redefining the integration variable as $s\, (1+x\, t+y\, t^{2} )=\sigma $ and using again \eqref{GrindEQ__1_}, yields the result reported in eq. \eqref{GrindEQ__5_}.
  \end{exmp}

 It is evident that the procedure we have outlined can easily be extended to generalized forms of Chebyshev polynomials $U_{n}^{(m)} (x,y)$ as it shown in the following example.
 
 \begin{exmp}
Through the use of lacunary Hermite \cite{DattLoren} and of its generating function
 
\begin{equation} \begin{split}\label{GrindEQ__9_} 
& H_{n}^{(m)} (x,y)=n! \sum _{r=0}^{\lfloor\frac{n}{m} \rfloor}\frac{x^{n-m\, r} y^{r} }{(n-m\, r)!\, r!}  ,\\ 
& \sum _{n=0}^{\infty }\frac{t^{n} }{n!}  H_{n}^{(m)} (x,y)=e^{x\, t+y\, t^{m} } ,\qquad \quad m>2,
 \end{split} \end{equation} 
we can introduce the lacunary Legendre (or Humbert) polynomials \cite{PHumbert,Boas}

\begin{equation}\begin{split} \label{GrindEQ__10_} 
& U_{n}^{(m)} (x,y)=\frac{1}{n!} \int _{0}^{\infty }e^{-s}  H_{n}^{(m)} (-s\; x,-s\;y)ds =
(-1)^{n} \sum _{r=0}^{\lfloor\frac{n}{m} \rfloor}\frac{(-1)^{(m-1)r} (n-(m-1)\; r)!\; x^{n-mr} y^{r} }{(n-m\; r)!\; r!}  , \\[1.2ex] 
& \sum _{n=0}^{\infty }t^{n}  U_{n}^{(m)} (x,y)=\frac{1}{1+x\; t+y\; t^{m} },\qquad \quad \forall m\geqslant 2, \;\; Re(1+x\; t+y\; t^{m} )>0.
 \end{split} \end{equation} 
  \end{exmp}
For this more general case too, the use of the properties of lacunary $HP$ can be usefully exploited to explore those of Lacunary Chebyshev polynomials.

\begin{exmp}
 To this aim we note e.g. that the ``Rainville'' generating function \cite{DattLoren}

\begin{equation} \begin{split}\label{GrindEQ__11_} 
& \sum _{n=0}^{\infty }\frac{t^{n} }{n!}  H_{n+l}^{(m)} (x,y)=H_{l}^{(m,m-1,...,1)} \left(\left\{\frac{p_{m}^{(n)} (x,y;t)}{n!} \right\}_{n=1,...,m} \right)e^{p_{m} (x,y;\, t)} ,\qquad \quad  \forall m,l\in\mathbb{N},\\[1.2ex] 
& H_{n}^{(p,p-1,...,1)} (x_{1} ,...,x_{p} )=n!\sum _{r=0}^{\lfloor\frac{n}{p}\rfloor}\frac{H_{n-p\, r}^{(p-1,p-2,...,1)} (x_{1} ,...,x_{p-1} )\, x_{p}^{r} }{(n-p\, r)!r!}, \\[1.2ex] 
& p_{m} (x,y;t)=x\, t+y\, t^{m} , \qquad \quad p_{m}^{(n)} (x,y; t)=\partial _{t}^{n} p_{m} (x,y;t),\quad n\le m , 
\end{split} \end{equation} 
where $H_{n}^{(p,p-1,...,1)} (x_{1} ,...,x_{p})$ are $p$-variable complete (non lacunary) $HP$ with generating function \cite{Babusci}

\begin{equation} \label{GrindEQ__12_} 
\sum _{n=0}^{\infty }\frac{t^{n} }{n!}  H_{n}^{(p,p-1,...,1)} (x_{1} ,...,x_{p} )=e^{\sum _{s=1}^{p}x_{s} t^{s}  } , 
\end{equation} 
yields, $\forall \mid t \mid <1,\;\; Re(1+x\, t+y\, t^{2})>0$,

\begin{equation}\label{GrindEQ__13_} 
 \sum _{n=0}^{\infty }t^{n}  \frac{(n+l)!}{n!} U_{n+l}^{(m)} (x,y)=l!
\frac{U_{l}^{(m,m-1,...)} \left(\frac{p_{m}^{(1)} (x,y;t)}{1+p_m(x,y;t)} ,\, \frac{1}{2} \frac{p_{m}^{(2)} (x,y;\, t)}{1+p_m(x,y;t)} ,...,\frac{1}{m!} \frac{p_{m}^{(m)} (x,y;t)}{1+p_m(x,y;t)} \right)}{1+p_{m} (x,y;t)}, 
\end{equation} 
where the complete $p$-variable Chebyshev polynomials are specified by means of the Laplace transform

\begin{equation}\label{key}
U_{n}^{(p,p-1,...,1)} (x_{1} ,...,x_{p} )=\frac{1}{n!} \int _{0}^{\infty }e^{-s}  H_{n}^{(p,p-1,...,1)} (-x_{1} s,...,-x_{p} s)ds ,
\end{equation}                    
straightforwardly yielding the generating function

\begin{equation} \label{GrindEQ__15_} 
\sum _{n=0}^{\infty }t^{n}  U_{n}^{(p,p-1,...,1)} (x_{1} ,...,x_{p} )=\frac{1}{1+\sum _{s=1}^{p}x_{s} t^{s}  }, \qquad \quad Re\left( 1+\sum _{s=1}^{p}x_{s} t^{s} \right)>0.  
\end{equation} 
\end{exmp}

The formalism associated with generalized Chebyshev polynomials is fairly flexible, therefore if we are interested in the successive derivatives of a rational function, we find the Proposition \ref{prop1}.

\begin{prop}\label{prop1}
	$\forall m\in\mathbb{N}$, $\forall x,y\in\mathbb{R}$, $\forall t\in\mathbb{R}^+:1+p_{2} (x,y;t)\neq 0$, a kind of Rodriguez formula for Chebyshev type polynomials holds
\begin{equation}\label{GrindEQ__16_} 
 \left( 1+p_{2} (x,y;\, t)\right) \partial _{t}^{m} \left(\frac{1}{1+p_{2} (x,y;t)} \right) =m!\;U_{m} \left(\frac{p_{2}^{(1)} (x,y;t)}{1+p_{2} (x,y;t)} ,\;\frac{1}{2} \frac{p_{2}^{(2)} (x,y;t)}{\left( 1+p_{2} (x,y;t)\right) } \right)
\end{equation} 	.
\end{prop}	
\begin{proof}	
	$\forall m\in\mathbb{N}$, $\forall x,y\in\mathbb{R}$, $\forall t\in\mathbb{R}^+:1+p_{2} (x,y;t)\neq 0$, 
\begin{equation}\begin{split} \label{GrindEQ__17_} 
 &\partial _{t}^{m} \left(\frac{1}{1+p_{2} (x,y;t)} \right)=\partial _{t}^{m} \int _{0}^{\infty }e^{-s} e^{-s\, x\, t-s\, y\, t^{2} }  ds=\\[1.2ex]
 & =\int _{0}^{\infty }e^{-s}  H_{m} \left( -sp_{2}^{(1)} (x,y;t)\, ,\, -\frac{s}{2} p_{2}^{(2)} (x,y;t)\right)e^{-s\, x\, t-s\, y\, t^{2} }  \, ds= \\[1.2ex] 
& =m!\frac{U_{m} \left(\frac{p_{2}^{(1)} (x,y;t)}{1+p_{2} (x,y;t)} ,\;\frac{1}{2} \frac{p_{2}^{(2)} (x,y;t)}{\left( 1+p_{2} (x,y;t)\right) } \right)}{1+p_{2} (x,y;\, t)} 
 \end{split} \end{equation} 
%
\end{proof}

\begin{obs}
Eq. \eqref{GrindEQ__16_}  which should be confronted with an analogous expression valid for the Higher order Hermite polynomials

\begin{equation} \label{GrindEQ__18_} 
\partial _{t}^{m} \left(e^{p_{n} (x,y;t)} \right)=H_{m}^{(n,n-1,...,1)} \left(\left\{\frac{p_{n}^{(s)} (x,y;t)}{s!} \right\}_{s=1,...,p} \right)e^{p_{n} (x,y;t)} . 
\end{equation} 
\end{obs}
The formalism we have just discussed shows how the properties of Hermite polynomials and of its generalized forms is a powerful tool to deal with other families of polynomials. In the forthcoming section we will go deeper into the relevant theory by discussing possible extension of the method to algebraic and umbral methods.

\section{ Umbral Methods and Chebyshev Polynomials }

The combination of algebraic and umbral methods has been shown to be extremely effective to deal with the properties (old and new) of special functions and polynomials.

\begin{rem}
To this aim we remind that the use of the umbral operator has allowed the rewriting of the Bessel functions as a Gaussian function, namely \cite{DBabusci,Babusci,SLicciardi}

\begin{equation}\label{J0}
J_0(x)=e^{-\hat{c}\left( \frac{x}{2}\right)^2 }\varphi_0,
\end{equation}
where $\hat{c}$ is an umbral operator, defined in such a way that acting on the vacuum $\varphi_0$ yields\footnote{See \cite{SLicciardi} for a rigorous treatment of the Umbral Methods.}

\begin{equation}\label{cOp}
\hat{c}^\nu \varphi_0=\dfrac{1}{\Gamma(\nu +1)}, \quad \forall \nu\in\mathbb{R}.
\end{equation}
According to the previous identities we recover, from eqs. \eqref{J0}-\eqref{cOp}, the Bessel series \cite{Dattoli}

\begin{equation}\label{key}
J_0(x)=\sum_{r=0}^{\infty}\dfrac{(-1)^r \left( \frac{x}{2}\right)^{2r} }{r!^2}.
\end{equation}
\end{rem}

Eq. \eqref{J0} states that the Gaussian is the umbral image of Bessel functions, however by further stretching the formalism we may also conclude that the Lorentz function can be used as the umbral image of the Bessel provided to set a new umbral operator $\hat{b}$.

\begin{defn}\label{Defopb}
We introduce the umbral operator $\hat{b}$ acting on the vacuum $\psi_0$ \cite{SLicciardi}

\begin{equation}\label{opb}
\hat{b}^\nu\;\psi_0=\dfrac{1}{\left( \Gamma(\nu+1)\right)^2 }.
\end{equation}
\end{defn}

\begin{prop}
	$\forall x\in\mathbb{R}$

\begin{equation}\label{key}
J_0(x)=\dfrac{1}{1+\hat{b}\left( \frac{x}{2}\right)^{2}}\psi_0.
\end{equation}
\end{prop}
\begin{proof}
$\forall x\in\mathbb{R}$
\begin{equation}\label{key}
J_0(x)=\sum_{r=0}^{\infty}\dfrac{(-1)^r \left( \frac{x}{2}\right)^{2r} }{r!^2}=
\sum_{r=0}^{\infty}\dfrac{(-1)^r \left( \frac{x}{2}\right)^{2r} }{\Gamma(r+1)^2}=
\sum_{r=0}^{\infty}(-\hat{b})^r \left( \frac{x}{2}\right)^{2r} \psi_0=\dfrac{1}{1+\hat{b}\left( \frac{x}{2}\right)^{2}}\psi_0.
\end{equation}	
\end{proof}	

\begin{cor}\label{cor1}
It is accordingly evident that $\forall x,\alpha,\beta\in\mathbb{R}$ 

\begin{equation}\label{key}
J_0\left( 2\sqrt{\alpha\;x +\beta\;x^2}\right)=\dfrac{1}{1+\hat{b}\left(\alpha\;x +\beta\;x^2 \right) }\psi_0 .
\end{equation}
\end{cor}

\begin{cor}\label{cor2}
By using the polynomial expansion in eq. \eqref{GrindEQ__5_}, we end up with

\begin{equation}\label{key}
J_0\left( 2\sqrt{\alpha\;x +\beta\;x^2}\right)=\sum_{n=0}^{\infty}x^n\;U_n\left(\alpha\; \hat{b},\beta\;\hat{b} \right)\psi_0. 
\end{equation}
\end{cor}

\begin{prop}
	$\forall x,\alpha,\beta\in\mathbb{R}$
	
\begin{equation}\begin{split}\label{key}
& J_0\left( 2\sqrt{\alpha\;x +\beta\;x^2}\right)=\sum_{n=0}^{\infty}x^n\;{}_2U_n\left(\alpha,\beta\right), \\
& {}_2U_n\left(\alpha,\beta\right):= (-1)^n\sum_{r=0}^{\lfloor\frac{n}{2}\rfloor}\dfrac{(-1)^r  \alpha^{n-2r}\beta^r}{(n-2r)!(n-r)!r!}.
\end{split}\end{equation}	
\end{prop}
\begin{proof}
	$\forall x,\alpha,\beta\in\mathbb{R}$, by using Definition \ref{Defopb} and Corollaries \ref{cor1} and \ref{cor2}, we find
\begin{equation}\label{key}
U_n\left(\alpha\; \hat{b},\beta\;\hat{b} \right)\psi_0=(-1)^n\sum_{r=0}^{\lfloor\frac{n}{2}\rfloor}\dfrac{(-1)^r (n-r)! \alpha^{n-2r}\beta^r}{(n-2r)!r!}\hat{b}^{n-r}\psi_0 =
(-1)^n\sum_{r=0}^{\lfloor\frac{n}{2}\rfloor}\dfrac{(-1)^r  \alpha^{n-2r}\beta^r}{(n-2r)!(n-r)!r!}.
\end{equation}
\end{proof}

\begin{obs}
It is also useful to note that

\begin{equation}\label{key}
\partial_{x}^m \;J_0\left( 2\sqrt{\alpha\;x +\beta\;x^2}\right)=\sum_{n=0}^{\infty}\dfrac{(n+m)!}{n!}x^n\;{}_2U_{n+m}\left(\alpha,\beta\right).
\end{equation}
\end{obs}

We have shown that the method we have proposed allows  a great deal of flexibility, in the sense that, hardly achievable results with ordinary means, are easily derived within the present computational framework. In the forthcoming section we will discuss further applications of this technique by exploring other properties of the polynomials belonging to the Legendre family.

\section{Legendre and Legendre-like Polynomials}

 It is almost natural to note that the polynomials defined through the integral representation \cite{L.C.Andrews} 
 
 \begin{equation}\begin{split}\label{Pn}
 P_{n} (x,y)&=\dfrac{(-1)^{n}}{\sqrt{\pi}} \sum _{r=0}^{\lfloor\frac{n}{2}\rfloor }\frac{\Gamma \left(n-r+\frac{1}{2} \right)x^{n-2r} (-y)^{r} }{(n-2\, r)!\, r!}  = \\
 & =\frac{1}{n!\Gamma \left(\frac{1}{2} \right)} \int _{0}^{\infty }e^{-s}  s^{-\frac{1}{2} } H_{n} (-s\, x,-sy)ds, \qquad \quad \forall x,y\in\mathbb{R},\forall n\in\mathbb{N}
\end{split} \end{equation}
are generated by

\begin{equation}\label{key}
\sum _{n=0}^{\infty }t^{n}  P_{n} (x,y)=\frac{1}{\sqrt{1+x\, t+y\, t^{2} } } 
\end{equation}
 and, upon replacing $x\to -2\, x,\, y=1$, are recognized as Legendre polynomials.

\begin{obs}
 All the results of the previous section can be naturally transposed to this family of polynomials thus finding e.g.

\begin{equation} \label{GrindEQ__13c_} 
\partial _{t}^{m} \dfrac{1}{\sqrt{1+p_{2} (x,y;t)} } =m!\dfrac{P_{m} \left(\dfrac{p_{2}^{(1)} (x,y;t)}{1+p_{2} (x,y;t)} ,\dfrac{1}{2!} \dfrac{p_{2}^{(2)} (x,y;t)}{1+p_{2} (x,y;t)} \right)}{\sqrt{1+p_{2} (x,y;t)} } , \quad \forall m\in\mathbb{N}   
\end{equation} 
or by its extension to the lacunary forms, namely 

\begin{equation}\label{key}
\partial _{t}^{n} \frac{1}{\sqrt{1+p_{m} (x,y;t)} } =n!\dfrac{P_{n} \left(\left\{\dfrac{1}{s!} \dfrac{p_{m}^{(s)} (x,y;t)}{1+p_{m} (x,y;t)} \right\}_{s=1,...,n} \right)}{\sqrt{1+p_{m} (x,y;t)} } .
\end{equation}    
\end{obs}

It is evident that the procedure we have envisaged is a fairly powerful tool to deal with analytical computation involving the calculation of successive derivatives of roots of polynomials.\\

 This is only one aspect of the question, we can further stretch the formalism and use the following identity \cite{Babusci}
 
\begin{equation} \label{GrindEQ__15b_} 
\lambda ^{x\partial _{x} } f(x)=f(\lambda x) 
\end{equation} 
to deduce, from \eqref{Pn}, the following operational definition of Legendre polynomials. 

\begin{prop}
	$\forall x,y\in\mathbb{R}$, $\forall n\in\mathbb{N}$,
	
\begin{equation} \begin{split}\label{GrindEQ__16b_} 
P_{n} (x,y)&=\frac{1}{n!\Gamma \left(\frac{1}{2} \right)} \int _{0}^{\infty }e^{-s}  s^{-\frac{1}{2} } (-s)^{(x\, \partial _{x} +y\, \partial _{y} )} ds\;H_{n} (\, x,y)=\\
& =\frac{(-1)^{(x\, \partial _{x} +y\, \partial _{y} )}}{\sqrt{\pi }n! }  \Gamma \left(\frac{1}{2} +(x\partial _{x} +y\partial _{y} )\right)H_{n} (x,y).
\end{split}\end{equation} 
\end{prop}

\begin{cor}
 Furthermore

\begin{equation} \label{GrindEQ__18b_} 
P_{n} (x,y)=\frac{(-1)^{(x\, \partial _{x} +y\, \partial _{y} )}}{\sqrt{\pi }n! }  \Gamma \left(\frac{1}{2} +(x\partial _{x} +y\partial _{y} )\right)\, e^{y\, \partial _{x}^{2} } x^{n}  .
\end{equation} 
\end{cor}
\begin{proof}
By recalling that the two variable $HP$ are a particular solutions of the heat equation and can be expressed as \cite{Babusci}	
	
\begin{equation} \label{GrindEQ__17b_} 
e^{y\, \partial _{x}^{2} } x^{n} =H_{n} (x,y) ,
\end{equation} 
we eventually end up with the result.
\end{proof}
Eq. \eqref{GrindEQ__18b_}  can be further exploited to develop a different point of view to the theory of Legendre like polynomials.

%

\section{On the Scaling Properties of the Legendre Polynomials}

In this section we embed integral transform methods and operational techniques, devloped so far, to derive and generalize some results concerning the Legendre polynomials $P_n(x)$ regarding the possibility of expressing $P_n(\lambda x)$, with $\lambda$ positive constant, as a sum of Legendre polynomials. The technique we propose is shown to be profitably extended to other families of special polynomials and we discuss the possibility of further research in this direction.\\

The Legendre polynomials are of a central importance in different branches of pure and applied Mathematics \cite{Arfken}. They play a pivotal role in the theory of approximation, integration and novel views to their properties often appear in current mathematical literature \cite{Olver,Bosch,Anli,Antonov,Szmytkowski,Bos}. Regarding some of the most recent investigations, new elements of discussion have been provided in two recent publications. In ref. \cite{DGMR} the integral representation method has been adopted to take advantage of the intimate connection between Legendre and Hermite polynomials, to provide an alternative point of view to the relevant theory. In ref. \cite{Laurent} new identities regarding the derivatives of Legendre polynomials and the associated multiplication theorems have been revisited from a different perspective, thus offering further insight to their properties. The strategy we follow is a byproduct of that introduced in the previous section and indeed we express the Legendre polynomials in terms of an integral representation involving Hermite polynomials.

\begin{rem}
To this aim we remind that the two-variable Legendre polynomials ($LP$) can be expressed through the series

\begin{equation}\label{key}
P_n(x,y)=\dfrac{(-1)^n}{\sqrt{\pi}}\sum_{r=0}^{\lfloor\frac{n}{2}\rfloor}\dfrac{\Gamma\left( n-r+\frac{1}{2}\right)x^{n-2r}(-y)^r }{(n-2r)!r!}, \quad \forall x,y\in\mathbb{R}, \forall n\in\mathbb{N}
\end{equation}
or by means of the integral representation \cite{L.C.Andrews,SLicciardi} 

\begin{equation}\label{eqPH}
P_n(x,y)=\dfrac{1}{\sqrt{\pi}n!}\int_0^\infty e^{-s}s^{-\frac{1}{2}}H_n(-s x,-s y)ds.
\end{equation}
%

\noindent The use of two-variable $HP$ generating function \eqref{GrindEQ__2_}
%
straightforwardly yields that associated with the two-variable $LP$, namely 

\begin{equation}\label{key}
\sum_{n=0}^\infty t^nP_n(x,y)=\dfrac{1}{\sqrt{1+xt+yt^2}}, 	\quad \left\lbrace  x,y,t\in \mathbb{R}: \left| \mid xt\mid+\mid y\mid t^2 \right| <1,  \right\rbrace .
\end{equation}
It is accordingly evident that the ordinary Legendre are identified as a particular case of two-variable $LP$  \cite{EWW}, namely 

\begin{equation}\label{key}
P_n(-2x,1)=P_n(x).
\end{equation}
where $(x\rightarrow-2x,\; y\rightarrow1)$ and
	\begin{equation}
		 P_n(x)=\dfrac{1}{2^n}\sum_{r=0}^{\lfloor\frac{n}{2}\rfloor}\dfrac{(-1)^r (2(n-r))! x^{n-2r}}{(n-2r)!(n-r)!r!}.
\end{equation}
	According to the definition of ${}_2U_n(\alpha,\beta)$ polynomials, we can easily infer that $P_n(x)$ can also be written as 
	
	\begin{equation}\label{key}
		P_n(x)=\left( -\dfrac{1}{2}\right)^n\int_{0}^\infty e^{-s}s^n\; {}_2U_n\left(xs,1 \right)ds  .        
	\end{equation}
\end{rem}
The properties of $H_n(x,y)$, summarized below, are extremely useful to establish those of the Legendre counterpart (for further comments see ref. \cite{Babusci})\footnote{Referring to the notation of the previous sections, we note that $H_n(x,y)=H_n^{(2)}(x,y)$ and it should be noted that analogous identities holds for their $H_n^{(m)}(x,y)$ counterparts.}.

\begin{propert}
\begin{enumerate}
	\item Variable dilatation
	\begin{equation}\label{VD}
	a^n H_n(x,y)=H_n(ax,a^2y);
	\end{equation}
	\item 	Repeated derivatives
	\begin{equation}\label{RD}
	\partial_x^r H_n(x,y)=\dfrac{n!}{(n-r)!}H_{n-r}(x,y);
	\end{equation}
	\item 	Multiplication Theorem
	\begin{equation}\label{MT}
	H_n(\lambda x,y)=\sum_{r=0}^n \left((\lambda-1)x \right)^r \binom{n}{r} H_{n-r}(x,y);
	\end{equation}
	\item 	Operational Definition
	\begin{equation}\label{OD}
	H_n(x,y)=e^{y\;\partial_x^2}x^n.
	\end{equation}
\end{enumerate}      
\end{propert}

\begin{cor}                                     
According to eq. \eqref{VD} we write

\begin{equation}\label{key}
P_n(\lambda x)=\dfrac{1}{\sqrt{\pi}n!}\int_0^\infty e^{-s}s^{n-\frac{1}{2}}H_n(2\lambda x,-s^{-1})\;ds, \qquad \forall \lambda,x\in\mathbb{R}, \forall n\in\mathbb{N}.
\end{equation}
The use of  eq. \eqref{MT} yields\footnote{Or, in a slightly different form,  \begin{equation*}\label{key}
	P_n(\lambda x)=\dfrac{1}{\sqrt{\pi}n!}\sum_{r=0}^n  \binom{n}{r}\left((2\lambda-1)x \right)^r \int_0^\infty e^{-s}s^{n-\frac{1}{2}}H_{n-r}( x,-s^{-1})\; ds.
	\end{equation*}} 

\begin{equation}\label{key}
P_n(\lambda x)=\dfrac{1}{\sqrt{\pi}n!}\sum_{r=0}^n  \binom{n}{r}\left((\lambda-1)2x \right)^r \int_0^\infty e^{-s}s^{n-\frac{1}{2}}H_{n-r}(2 x,-s^{-1})\;ds.
\end{equation}
According to the successive derivative property (eq. \eqref{RD}) we end up with\footnote{Or \begin{equation*}\label{key}
	P_n(\lambda x)=\dfrac{1}{\sqrt{\pi}n!}\sum_{r=0}^n \dfrac{ \left((2\lambda-1)x \right)^r}{r!}  \int_0^\infty e^{-s}s^{n-\frac{1}{2}}\partial_x^r H_n\left( x,-\dfrac{1}{s}\right)\; ds .
	\end{equation*}} 

\begin{equation}\label{key}
P_n(\lambda x)=\dfrac{2^n}{\sqrt{\pi}n!}\sum_{r=0}^n \dfrac{ \left((\lambda-1)x \right)^r}{r!}  \int_0^\infty e^{-s}s^{n-\frac{1}{2}}\partial_x^r H_n\left( x,-\dfrac{1}{4\;s}\right) ds
\end{equation}
which, after a trivial rearrangement, allows the final result

\begin{equation}\label{key}
P_n(\lambda x)=\sum_{r=0}^n \dfrac{(\lambda -1)^r}{r!}x^r\partial_x^r P_n(x).
\end{equation}
\end{cor} 

\begin{cor}
According to eq. \eqref{OD} it is evident that the following operational identity holds true

\begin{equation}
\begin{split}
& P_n(x,y)=\hat{\Delta}\left( \dfrac{(-x)^n}{n!}\right) ,\\
& \hat{\Delta}=\dfrac{1}{\sqrt{\pi}}\int_0^\infty e^{-s}s^{n-\frac{1}{2}}e^{-\frac{y}{s}\partial_x^2}ds.
\end{split} 
\end{equation}  
\end{cor} 

\begin{prop}
	$\forall \lambda,x,y\in\mathbb{R}$, $\forall n\in\mathbb{N}$
	
\begin{equation}\label{Pnl}
P_n(\lambda x,y)=\sum_{r=0}^n \dfrac{y^r}{r!}\left(1-\lambda^2 \right)^r \lambda^{n-2r}\partial_x^r P_{n-r}(x,y).
\end{equation} 		
\end{prop}
\begin{proof}
	The use of the same procedure of previuous corollaries yields

\begin{equation}\label{key}
P_n(\lambda x,y)=\dfrac{1}{\sqrt{\pi}}\int_0^\infty e^{-s}s^{n-\frac{1}{2}}e^{\frac{y}{s}\left( \frac{\lambda^2 -1}{\lambda^2}\right) \partial_x^2}e^{-\frac{y}{s}\partial_x^2}ds\left(\dfrac{(-\lambda x)^n}{n!} \right) 
\end{equation}
which, after expanding the exponential inside the integral, allows the following conclusion

\begin{equation}\label{key}
\begin{split}
P_n(\lambda x,y)&=
\dfrac{\lambda^n}{\sqrt{\pi}}\sum_{r=0}^n \dfrac{y^r}{r!}\left( \frac{\lambda^2 -1}{\lambda^2}\right)^r \int_0^\infty e^{-s}s^{n-r-\frac{1}{2}}\partial_x^{2r}\left( e^{-\frac{y}{s}\partial_x^2}\left(\dfrac{(- x)^n}{n!} \right) \right) ds=\\
& =\dfrac{\lambda^n}{\sqrt{\pi}n!}\sum_{r=0}^n \dfrac{y^r}{r!}\left( \frac{\lambda^2 -1}{\lambda^2}\right)^r \int_0^\infty e^{-s}s^{n-r-\frac{1}{2}}\partial_x^{2r} H_n\left( -x,-\dfrac{y}{s}\right) ds=\\
& = \dfrac{\lambda^n}{\sqrt{\pi}}\sum_{r=0}^n \dfrac{y^r}{r!}\left( \frac{1-\lambda^2 }{\lambda^2}\right)^r \dfrac{1}{(n-r)!}\int_0^\infty e^{-s}s^{n-r-\frac{1}{2}}\partial_x^{r} H_{n-r}\left( -x,-\dfrac{y}{s}\right) ds
\end{split} 
\end{equation}   
which finally yields Eq. \eqref{Pnl}.
\end{proof}             
The same identity has been previously derived by the use of more conventional means in ref. \cite{Laurent}.\\

The results we have obtained are a consequence of the integral representation yielding a link between Legendre and Hermite polynomials. We will complete this paragraph by considering further relations of the scaling type corroborating the importance of this representation.\\

In ref. \cite{SLicciardi} the following umbral representation of Hermite polynomials has been shown to be particularly useful

\begin{equation} \label{HermLagbisScaling}
H_{n} (x,y)  = \left( x+{}_y {\hat{h}}\right) ^n \theta_0, \quad \forall x,y\in\mathbb{R},\forall n\in\mathbb{N}.
\end{equation}
Within such a context, the umbral operator ${}_y\hat{h}$  acts on the vacuum $\theta_0$ according to the rule\footnote{The notion of vacuum for umbral operators can be found in \cite{SLicciardi} and a convenient representation of the operator and of the Hermite function vacuum is \begin{equation*}\label{key}
	\theta(z) :=\theta_z=y^{\frac{z}{2}}
	\left( \dfrac{\Gamma(z+1)}{\Gamma\left( \frac{1}{2}z+1\right) }\left|  \cos \left( \frac{\pi}{2}z\right)  \right| \right) , \quad \forall z\in\mathbb{R}.
	\end{equation*}} 	

\begin{equation}
\begin{split}\label{eq2HermLagbis} 
& {}_y \hat{h}^r\;\theta_0 :=\theta_r, \quad \forall r\in \mathbb{R},\\
& \theta_r =\dfrac{y^{\frac{r}{2}}r!}{\Gamma\left( \frac{r}{2}+1\right) }\left|  \cos \left( r\dfrac{\pi}{2}\right) \right| = 
\left\lbrace   \begin{array}{ll}
0                           & r=2s+1 \\
y^s \dfrac{(2s)!}{s!} & r=2s
\end{array}\right. \;\;\forall s\in\mathbb{Z}.
\end{split}\end{equation}  
It is accordingly evident that

\begin{equation}\label{yhgenfunScaling}
e^{{}_{y}\hat{h}t}\theta_{0}=\sum_{r=0}^{\infty}\dfrac{t^r}{r!} {}_{y}\hat{h}^{\;r}\theta_{0}=e^{yt^2},\quad \forall y,t\in\mathbb{R}.
\end{equation}

\begin{prop}
	$\forall x,y\in\mathbb{R}$
	
\begin{equation}\label{key}
\lim_{n\rightarrow\infty} H_n\left(x, \dfrac{y}{n^2} \right)\simeq x^n e^{\frac{y}{x^2}}.
\end{equation}	
\end{prop}

\begin{proof}
Along with Eq. \eqref{HermLagbisScaling}, Eq. \eqref{yhgenfunScaling} yields 

\begin{equation}\label{key}
\lim_{n\rightarrow\infty} H_n\left(x, \dfrac{y}{n^2} \right)=
\lim_{n\rightarrow\infty} \left(x+\dfrac{{}_y\hat{h}}{n}\right) ^n\theta_{0}=
\lim_{n\rightarrow\infty}x^n\left(1+ \dfrac{{}_y\hat{h}}{nx}\right) ^n\simeq x^n e^{\frac{{}_y\hat{h}}{x}}\theta_{0}=x^n e^{\frac{y}{x^2}}.
\end{equation}
\end{proof}

\begin{cor}
If we use the above asymptotic limit and the integral representation in eq. \eqref{eqPH} we obtain


\begin{equation}\label{key}
\lim_{n\rightarrow\infty}P_n\left(x, \dfrac{y}{n^2} \right)\simeq \dfrac{(-x)^n}{\sqrt{\pi}\;n!}\int_0^\infty e^{-\left(s+\frac{y}{x^2s} \right) }s^{n-\frac{1}{2}}ds.
\end{equation}
\end{cor}

Let us finally note that the Gegenbauer polynomials \cite{L.C.Andrews,Abramovitz}, specified by the generating function 

\begin{equation}\label{key}
\sum_{n=0}^\infty t^n C_n^{(\gamma)}(x)=\dfrac{1}{\left(1-2x t +t^2 \right)^\gamma }\;\;, \qquad  \quad \mid -2x t\mid +t^2<1,
\end{equation}
can be expressed in terms of two variable Hermite as \cite{DGMR} 

\begin{equation}\label{Cngamm}
C_n^{(\gamma)}(x)=\dfrac{1}{\Gamma(\gamma)\;n!}\int_0^\infty e^{-s}s^{\gamma - 1}H_n(2 s x, -s)\;ds.
\end{equation}

By using the same arguments as before we can infer the further following identity.

\begin{prop}
	$\forall \lambda,x\in\mathbb{R}$, $\forall n\in\mathbb{N}$, $\forall \gamma\in\mathbb{R}^+$ 
\begin{equation}\label{key}
C_n^{(\gamma)}(\lambda x)=\sum_{r=0}^n \dfrac{(\lambda^2 -1)^r \lambda^{n-2r}}{r!\;2^r} \partial_x ^r C_{n-r}^{(\gamma)}(x).
\end{equation}
\end{prop}

\begin{obs}
The well-known asymptotic properties \cite{Temme}

\begin{equation}\label{key}
\lim_{\gamma\rightarrow\infty}\left( \gamma^{-\frac{n}{2}}C_n^{(\gamma)}\left(\dfrac{x}{\sqrt{\gamma}} \right) \right)=\dfrac{1}{n!}H_n (x) 
\end{equation}
is just a consequence of eq. \eqref{Cngamm}, we find indeed

\begin{equation}\label{key}
\begin{split}
\lim_{\gamma\rightarrow\infty}\left( \gamma^{-\frac{n}{2}}C_n^{(\gamma)}\left(\dfrac{x}{\sqrt{\gamma}} \right) \right)&=
\dfrac{1}{n!}\lim_{\gamma\rightarrow\infty}\int_0^\infty e^{-s}\dfrac{s^{\gamma - 1}}{\Gamma(\gamma)}\gamma^{-\frac{n}{2}}H_n\left(2s\dfrac{x}{\sqrt{\gamma}},-s \right)ds=\\
& =
\dfrac{1}{n!}\lim_{\gamma\rightarrow\infty}\int_0^\infty e^{-s}\dfrac{s^{\gamma - 1}}{\Gamma(\gamma)} 
H_n\left(2s\dfrac{x}{\gamma},\dfrac{-s}{\gamma} \right)ds\simeq
\dfrac{1}{n!}H_n(2x,-1).
\end{split}
\end{equation}
\end{obs}

In this paper we have provided an account of operational properties allowing a general tool to study various properties of ordinary and generaized polynomials in a unified and straightforward framework. A forthcoming publication will be addressed to a more systematic use applied to wider classes of special functions.\\

\textbf{Acknowledgements}\\

The work of Dr. S. Licciardi was supported by an Enea Research Center individual fellowship.\\
Prof. R.M. Pidatella wants to thank the fund of University of Catania "Metodi gruppali e umbrali per modelli di diffusione e trasporto" for partial support of this work.\\

\textbf{Author Contributions}\\

Conceptualization: S.L., G.D.; methodology: S.L., G.D.; data curation: S.L.; validation: S.L., G.D., R.M.P.; formal analysis: S.L., G.D.; writing - original draft preparation: S.L., G.D.; writing - review and editing: S.L. .\\

{}


\begin{thebibliography}{}
	
  \bibitem{Dattoli} Dattoli, G.; Ottaviani, P.L.; Torre, A. and Vazquez, L. Evolution Operator Equations: Integration with Algebraic and Finite Difference Methods. Application to Physical Problems in Classical and Quantum Mechanics and Quantum Field Theory, \textbf{1997}, \textit{La Rivista Del Nuovo Cimento}, \textit{20(2),} 3--133.
  
  \bibitem{S.Roman} Roman, S. The Umbral Calculus, Dover Publications, New York, 2005. 
  
  \bibitem{L.C.Andrews} Andrews, L.C. Special Functions For Engeneers and Applied mathematicians, Mc Millan, New York, 1985.
  
 \bibitem{Babusci} Babusci D., Dattoli G., Licciardi S., Sabia E. Mathematical Methods for Physics,  World Scientific, Singapore, 2019.
  
  \bibitem{DattLoren} Dattoli, G.; Lorenzutta, S. and Cesarano, C. From Hermite to Humbert Polynomials, \textbf{2003}, \textit{Rend. Istit. Mat. Univ. Trieste}, Vol. XXXV, 37--48.
  
  \bibitem{Appel} App\'el, P.;  Kamp\'e de F\'eri\'et, J. Fonctions Hypergeometriques and Hyperspheriques. Polynomes d'Hermite, Gauthiers-Villars, Paris, 1926.
  
  \bibitem{PHumbert} Humbert, P.; Some extensions of Pincherle's Polynomials, Proceedings of the Edinburgh Mathematics Society, \textbf{1921}, \textit{39}, 21--24.
  
  \bibitem{Boas} Boas, R.P.; Buck, R.C. Polynomial expansions of analytic functions, Ergebnisse der Mathematik und ihrer Grenzgebiete. Neue Folge., 19, Berlin, New York: Springer-Verlag, MR 0094466, 1958.
  
  \bibitem{DBabusci} Babusci, D.; Dattoli, G.; Gorska, K.; Penson, K. Symbolic methods for the evaluation of sum rules of Bessel functions, \textit{J. Math. Phys.}, \textbf{2013}, \textit{54}, 073501.
  
  \bibitem{SLicciardi} Licciardi, S. Umbral Calculus, a Different Mathematical Language [Ph D Thesis].
  University of Catania, 2018,  arXiv:1803.03108 [math.CA].	
  
  
  
  \bibitem{Arfken} Arfken, G.B. Mathematical methods for physicists, 3rd ed. Orlando Academic Press, 1985.
  
  \bibitem{Olver} Olver, F.W.J.; Lozier, D.W.; Boisvert, R.F.; Clark, C.W. NIST Handbook of
  	Mathematical Functions, Cambridge University Press, 2010.
  
  \bibitem{Bosch} Bosch, W. On the Computation af Derivatives of Legendre Functions, \textit{Phys. Chem. Earth} (A), \textbf{2000}, \textit{25(655)}.
  
  \bibitem{Anli} Anli, F.; Gungor, S. Some useful properties of Legendre polynomials and its applications to neutron transport equation in slab geometry, \textit{Appl. Mathem. Model.}, \textbf{2007}, \textit{31(727)}.
  
  \bibitem{Antonov} Antonov, V.A.; Kholshevnikov, K.V.;  Shaidulin, V.Sh. Estimating the Derivative
  	of the Legendre Polynomial, \textit{Mathematics}, \textbf{2010}, \textit{43(191)}, Vestnik St. Petersburg University.
  
  \bibitem{Szmytkowski} Szmytkowski, R. On the derivative of the associated Legendre function of the first kind of integer order with respect to its degree (with applications to the construction of the
  	associated Legendre function of the second kind of integer degree and order), \textit{J.
  Math. Chem.}, \textbf{2011}, \textit{49(1436)}.
  
  \bibitem{Bos} Bos, L.; Narayan, A.; Levenberg, N.; Piazzon, F. An Orthogonality Property of
  	the Legendre Polynomials, \textit{Constr. Approx.}, \textbf{2017}, \textit{45(65)}.
  
  \bibitem{DGMR} Dattoli, G.; Germano, B.; Martinelli, M.R.; Ricci, P.E. A novel theory of Legendre Polynomials, \textit{Math. Comput. Model.}, \textbf{2011}, \textit{54}, 80--87.
  
  
  \bibitem{Laurent} Laurent, G.M.; Harrison, G.R. The scaling properties and the multiple derivative of Legendre polynomials, arXiv:1711.00925v1 [mat.CA] 28-Oct-\textbf{2017}.
  
  \bibitem{EWW} Weisstein, E. Legendre Polynomial, From MathWorld--A Wolfram Web Resource http://mathworld.wolfram.com/about/author.html.
  
  \bibitem{Abramovitz} Abramovitz, M. and Stegun, I.A. Parabolic Cylinder Function, Ch. 19 in Handbook of Mathematical Functions with Formulas, Graphs and Mathematical Tables, 9th printing, New York: Dover, 1972, 685--700. 
  
  \bibitem{Temme} Temme, N.M. Special Functions: An Introduction to the Classical Functions of Mathematical Physics, Wiley, New York, 1996.
  

  
\end{thebibliography}
\end{document}